\documentclass[11pt]{amsart}
\usepackage{amssymb}
\usepackage{amsmath}
\usepackage{mathrsfs}
\usepackage{amsthm}
\newtheorem{theorem}{Theorem}[section]
\newtheorem{lemma}[theorem]{Lemma}
\newtheorem{corollary}[theorem]{Corollary}
\theoremstyle{definition}
\newtheorem{problem}{Problem}
\newtheorem*{remark}{Remark}
\numberwithin{equation}{section}
\mathchardef\hyphen="2D
\begin{document}
\title{Calkin representations for $L^{p}$}
\author{March T.~Boedihardjo}
\address{Department of Mathematics, University of California, Los Angeles, CA 90095-1555}
\email{march@math.ucla.edu}
\begin{abstract}
We identify the weak closures of the ranges of certain Calkin representations for $L^{p}$, $1<p<\infty$. As a consequence, assuming the continuum hypothesis, we show that the commutant of $B(L^{p})$, $1<p<\infty$, in its ultrapower may or may not be trivial depending on the ultrafilter. This extends a result of Farah, Phillips and Stepr\=ans.
\end{abstract}
\maketitle
\section{Introduction}
Throughout this paper, the scalar field can be either $\mathbb{R}$ or $\mathbb{C}$. Given a Banach space $\mathcal{X}$, let $B(\mathcal{X})$ be the algebra of operators on $\mathcal{X}$ and let $K(\mathcal{X})$ be the ideal of compact operators on $\mathcal{X}$. In \cite{Calkin}, Calkin began the study of the quotient algebra $B(l^{2})/K(l^{2})$ and explicitly constructed a class of isometric representations of $B(l^{2})/K(l^{2})$ on nonseparable Hilbert spaces. In \cite{Reid}, Reid showed that certain Calkin representations are irreducible and in fact specified exactly which. Reid's result provides the first known explicit example of an irreducible representation of $B(l^{2})/K(l^{2})$.

Calkin representations for Banach spaces were studied in \cite{Boedihardjo}. It was shown that certain Calkin representations for Banach spaces are bounded below. In this paper, we identify the weak closures of the ranges of certain Calkin representations for $L^{p}$, $1<p<\infty$. As a consequence, assuming the continuum hypothesis, we show that if $1<p<\infty$ then the commutant of $B(L^{p})$ in its ultrapower may or may not be trivial depending on the ultrafilter. For $p=2$, this was proved by Farah, Phillips and Stepr\=ans \cite{Farah}.

Let $\mathcal{X}$ be a reflexive Banach space. Let $l^{\infty}(\mathcal{X})$ be the space of bounded functions from $\mathbb{N}$ to $\mathcal{X}$. Let $\mathcal{U}$ be a nonprincipal ultrafilter on $\mathbb{N}$. The {\it ultrapower} $\mathcal{X}^{\mathcal{U}}$ of $\mathcal{X}$ with respect to $\mathcal{U}$ (see \cite{Diestel}) is the quotient of the Banach space $l^{\infty}(\mathcal{X})$ by the subspace
\[c_{\mathcal{U}}(\mathcal{X})=\{(x_{n})_{n\geq 1}\in l^{\infty}(\mathcal{X}):\lim_{n,\mathcal{U}}\|x_{n}\|=0\}.\]
If $(x_{n})_{n\geq 1}$ is a bounded sequence in $\mathcal{X}$, then its image in $\mathcal{X}^{\mathcal{U}}$ is denoted by $(x_{n})_{n,\mathcal{U}}$. It is easy to see that $\displaystyle\|(x_{n})_{n,\mathcal{U}}\|=\lim_{n,\mathcal{U}}\|x_{n}\|$. If $T\in B(\mathcal{X})$ then its {\it ultrapower} $T^{\mathcal{U}}\in B(\mathcal{X}^{\mathcal{U}})$ is defined by $(x_{n})_{n,\mathcal{U}}\mapsto(Tx_{n})_{n,\mathcal{U}}$.

Consider the subspace $\{(x)_{n,\mathcal{U}}:x\in\mathcal{X}\}$ of $\mathcal{X}^{\mathcal{U}}$. We shall identify this subspace with $\mathcal{X}$. The canonical projection from $\mathcal{X}^{\mathcal{U}}$ onto $\mathcal{X}$ is given by $\displaystyle(x_{n})_{n,\mathcal{U}}\mapsto w\hyphen\lim_{n,\mathcal{U}}x_{n}$, where $\displaystyle w\hyphen\lim_{n,\mathcal{U}}x_{n}$ is the weak limit of $(x_{n})_{n\geq 1}$ through $\mathcal{U}$ which exists since $\mathcal{X}$ is reflexive. Thus, $\mathcal{X}^{\mathcal{U}}$ admits the decomposition
\[\mathcal{X}^{\mathcal{U}}=\mathcal{X}\oplus\widehat{\mathcal{X}},\]
where
\[\widehat{\mathcal{X}}=\{(x_{n})_{n,\mathcal{U}}\in\mathcal{\mathcal{X}}^{\mathcal{U}}:w\hyphen\lim_{n,\mathcal{U}}x_{n}=0\}.\]
Note that both subspaces $\{(x)_{n,\mathcal{U}}:x\in\mathcal{X}\}$ and $\widehat{\mathcal{X}}$ of $\mathcal{X}^{\mathcal{U}}$ are invariant under $T^{\mathcal{U}}$ for all $T\in B(\mathcal{X})$. For each $T\in B(\mathcal{X})$, let $\widehat{T}\in B(\widehat{\mathcal{X}})$ be the restriction of $T^{\mathcal{U}}$ to $\widehat{\mathcal{X}}$. We have
\begin{equation}\label{tdecomp}
T^{\mathcal{U}}=T\oplus\widehat{T}
\end{equation}
with respect to the decomposition $\mathcal{X}^{\mathcal{U}}=\mathcal{X}\oplus\widehat{\mathcal{X}}$.

The map $T\mapsto\widehat{T}$ defines a linear homomorphism from $B(\mathcal{X})$ into $B(\widehat{\mathcal{X}})$. It is easy to see that if $K\in K(\mathcal{X})$ then $\widehat{K}=0$. Let $\pi:B(\mathcal{X})\to B(\mathcal{X})/K(\mathcal{X})$ be the quotient map. Define a homomorphism $\rho_{\mathcal{U}}:B(\mathcal{X})/K(\mathcal{X})\to B(\widehat{\mathcal{X}})$ by
\[\rho_{\mathcal{U}}(\pi(T))=\widehat{T},\quad T\in B(\mathcal{X}).\]
The homomorphism $\rho_{\mathcal{U}}$ is the {\it Calkin representation} for $\mathcal{X}$ with respect to $\mathcal{U}$. When $\mathcal{X}$ is a Hilbert space, Calkin showed that \cite{Calkin} $\rho_{\mathcal{U}}$ is an isometric $*$-representation. The author and Johnson showed that \cite{Boedihardjo} if $\mathcal{X}$ is reflexive and has the compact approximation property (in particular, if $\mathcal{X}$ is reflexive and has a Schauder basis), then
\[\frac{1}{2}\|\pi(T)\|\leq\|\rho_{\mathcal{U}}(\pi(T))\|\leq\|\pi(T)\|,\quad T\in B(\mathcal{X}).\]
Reid showed that \cite{Reid} when $\mathcal{X}=l^{2}$, the representation $\rho_{\mathcal{U}}$ is irreducible if and only if the ultrafilter $\mathcal{U}$ is selective. Thus, if $\mathcal{U}$ is selective then the range $\{\rho_{\mathcal{U}}(\pi(T)):T\in B(l^{2})\}$ of $\rho_{\mathcal{U}}$ is dense in $B(\widehat{l^{2}})$ in the weak operator topology (WOT). Throughout this paper, $\mu$ is the Lebesgue measure on $[0,1]$ and $L^{p}=L^{p}([0,1],\mu)$. The main result of this paper is
\begin{theorem}\label{main}
Let $1<p<\infty$, $p\neq 2$. Let $\mathcal{U}$ be a selective nonprincipal ultrafilter on $\mathbb{N}$. Then there is a nontrivial subspace $\mathcal{M}$ of $\widehat{L^{p}}$ such that the WOT closure of the range of $\rho_{\mathcal{U}}$ is given by
\[\{\rho_{\mathcal{U}}(\pi(T)):T\in B(L^{p})\}^{-WOT}=\{S\in B(\widehat{L^{p}}):S\mathcal{M}\subset\mathcal{M}\}.\]
\end{theorem}
In Section 2, the space $\mathcal{M}$ described in Theorem \ref{main} is given explicitly when $p>2$. We show that it is invariant under all $\rho_{\mathcal{U}}(\pi(T))$. In Section 3, we construct a projection onto $\mathcal{M}$. In Section 4, we prove some properties of the projection constructed in Section 3. In Section 5, we prove some lemmas that uses the selectivity of $\mathcal{U}$. In Section 6, we give the proof of Theorem \ref{main} and describe the space $\mathcal{M}$ when $p<2$. In Section 7, we give some consequences of Theorem \ref{main}. In Section 8, we state a few open problems.

We begin with some preliminaries.

Let $\mathcal{X}$ be a Banach space. If $(T_{n})_{n\geq 1}$ is a bounded sequence in $B(\mathcal{X})$, then its {\it ultraproduct} $(T_{1},T_{2},\ldots)_{\mathcal{U}}$ is the operator on $\mathcal{X}^{\mathcal{U}}$ defined as
\[(T_{1},T_{2},\ldots)_{\mathcal{U}}(x_{n})_{n,\mathcal{U}}=(T_{n}x_{n})_{n,\mathcal{U}},\quad (x_{n})_{n,\mathcal{U}}\in\mathcal{X}^{\mathcal{U}}.\]
It is easy to see that $\displaystyle\|(T_{1},T_{2},\ldots)_{\mathcal{U}}\|=\lim_{n,\mathcal{U}}\|T_{n}\|$.

A Banach space $\mathcal{X}$ is {\it superreflexive} if every Banach space $\mathcal{Y}$ finitely representable in $\mathcal{X}$ is reflexive, or equivalently, if every ultrapower of $\mathcal{X}$ is reflexive. If $1<p<\infty$ then $L^{p}$ is superreflexive \cite{Stern}. Stern showed that \cite[Theorem 2.3]{Stern} a Banach space $\mathcal{X}$ is superreflexive if and only if $(\mathcal{X}^{\mathcal{U}})^{*}=(\mathcal{X}^{*})^{\mathcal{U}}$, i.e., for every bounded linear functional $\phi$ on $\mathcal{X}^{\mathcal{U}}$, there exists a unique $(x_{n}^{*})_{n,\mathcal{U}}\in(\mathcal{X}^{*})^{\mathcal{U}}$ such that
\begin{equation}\label{duality}
\phi[(x_{n})_{n,\mathcal{U}}]=\lim_{n,\mathcal{U}}x_{n}^{*}(x_{n}),\quad (x_{n})_{n,\mathcal{U}}\in\mathcal{X}^{\mathcal{U}}.
\end{equation}
Thus for $1<p<\infty$, the dual of $(L^{p})^{\mathcal{U}}$ is $(L^{q})^{\mathcal{U}}$ where $\displaystyle\frac{1}{p}+\frac{1}{q}=1$. Under this identification, the dual of $\widehat{L^{p}}$ is $\widehat{L^{q}}$.

Let $(x_{n})_{n\geq 1}$ be a sequence in a Banach space $\mathcal{X}$. Let $(y_{n})_{n\geq 1}$ be a sequence in a Banach space $\mathcal{Y}$. The sequences $(x_{n})_{n\geq 1}$ and $(y_{n})_{n\geq 1}$ are {\it equivalent} \cite{Albiac} if there is a constant $C>0$ such that
\[\frac{1}{C}\left\|\sum_{n=1}^{r}a_{n}y_{n}\right\|\leq\left\|\sum_{n=1}^{r}a_{n}x_{n}\right\|\leq C\left\|\sum_{n=1}^{r}a_{n}y_{n}\right\|,\]
for all $r\geq 1$ and scalars $a_{1},\ldots,a_{r}$.

Let $(x_{n})_{n\geq 1}$ be a sequence in a Banach space $\mathcal{X}$. We say that the summation $\displaystyle\sum_{n=1}^{\infty}x_{n}$ {\it converges unconditionally} if the summation $\displaystyle\sum_{n=1}^{\infty}\epsilon_{n}x_{n}$ converges in norm for every $\epsilon_{1},\epsilon_{2},\ldots\in\{-1,1\}$. By completeness of $\mathcal{X}$, if for every $\epsilon>0$, there exists $N\geq 1$ such that
\[\left\|\sum_{n\in F}x_{n}\right\|<\epsilon,\]
for every $F\subset\mathbb{N}\cap[N,\infty)$, then the summation $\displaystyle\sum_{n=1}^{\infty}x_{n}$ converges unconditionally.

Let $(e_{j})_{j\geq 1}$ be a sequence in a Banach space $\mathcal{X}$. A sequence $(x_{n})_{n\geq 1}$ of the form
\[x_{n}=\sum_{j=k(n)}^{k(n+1)-1}a_{j}e_{j},\]
where $1=k(1)<k(2)<\ldots$ and $a_{1},a_{2},\ldots$ are scalars, is called a {\it block sequence} of $(e_{n})_{n\geq 1}$.

Let $1\leq p<\infty$. The sequence $(u_{j})_{j\geq 1}$ in $L^{p}$ defined by $u_{1}=1$ and
\[u_{2^{k}+r}(t)=\begin{array}{cc}\begin{cases}1,&\frac{2r-2}{2^{k+1}}\leq t<\frac{2r-1}{2^{k+1}}\\-1,&\frac{2r-1}{2^{k+1}}\leq t<\frac{2r}{2^{k+1}}\\0,&\text{Otherwise}\end{cases},\end{array}\]
where $k=0,1,2\ldots$ and $r=1,\ldots,2^{k}$, is the {\it Haar basis} for $L^{p}$ \cite{Albiac}.

If $x\in\mathcal{X}$ and $x^{*}\in\mathcal{X}^{*}$, then $x\otimes x^{*}$ is the rank one operator on $\mathcal{X}$ defined by $y\mapsto x^{*}(y)x$. If $A$ is a Borel set in $[0,1]$, the complement of $A$ in $[0,1]$ is denoted by $A^{c}$ and the indicator function of $A$ is denoted by $I(A)$. If $f:[0,1]\to\mathbb{R}$ is a measurable function, then the essential support of $f$ on $[0,1]$ is denoted by $\mathrm{supp}(f)$ and the $L^{p}$ norm of $f$ is denoted by $\|f\|_{p}$. The range of an operator $T$ is denoted by $\mathrm{ran}\,T$.

Let $\mathcal{B}$ be a Banach algebra. Let $\mathcal{A}$ be a subalgebra of $\mathcal{B}$. The commutant of $\mathcal{A}$ in $\mathcal{B}$ is the subalgebra
\[\mathcal{A}'\cap\mathcal{B}=\{b\in\mathcal{B}:ab=ba\text{ for all }a\in\mathcal{A}\}\]
of $\mathcal{B}$.

A set $A$ is {\it almost contained} in another set $B$ if $A\backslash B$ is finite. An ultrafilter $\mathcal{U}$ on $\mathbb{N}$ is {\it selective} (see \cite{Farah} or \cite{Reid} where the latter used the word {\it absolute}) if
\begin{enumerate}
\item for every sequence $A_{1},A_{2},\ldots$ of sets in $\mathcal{U}$, there exists $A\in\mathcal{U}$ that is almost contained in each $A_{k}$; and
\item given any partition of $\mathbb{N}$ into disjoint finite sets $A_{1},A_{2},\ldots$, there exists $A\in\mathcal{U}$ such that $A\cap A_{k}$ is a singleton for each $k$.
\end{enumerate}
A selective nonprincipal ultrafilter exists if we assume the continuum hypothesis (see \cite{Farah}).
\section{The space $\mathcal{M}$}
For $2<p<\infty$, the space $\mathcal{M}$ described in Theorem \ref{main} is given by
\begin{equation}\label{m}
\mathcal{M}=\left\{(f_{n})_{n,\mathcal{U}}\in(L^{p})^{\mathcal{U}}:\lim_{n,\mathcal{U}}\|f_{n}\|_{2}=0\right\}.
\end{equation}
It is easy to see that $\mathcal{M}$ is a closed linear subspace of $\widehat{L^{p}}$. The space $\mathcal{M}$ is nontrivial since $(I(A_{n})/\|I(A_{n})\|_{p})_{n,\mathcal{U}}\in\mathcal{M}$ for every sequence $(A_{n})_{n\geq 1}$ of sets in $[0,1]$ such that $\mu(A_{n})\to 0$. In this section, we show that $\mathcal{M}$ is invariant under $\rho_{\mathcal{U}}(\pi(T))$ for all $T\in B(L^{p})$.

To begin, let us recall a classical result of Kadec and Pe\l{}czy\'nski.
\begin{lemma}[\cite{Kadec}]\label{Kad}
Let $2<p<\infty$. If $(f_{n})_{n\geq 1}$ is a sequence in $L^{p}$ converging to 0 weakly, then there is a subsequence $(f_{n_{k}})_{k\geq 1}$ satisfying either
\begin{enumerate}
\item $\displaystyle\lim_{k\to\infty}\|f_{n_{k}}\|_{p}=0$;
\item $(f_{n_{k}})_{k\geq 1}$ is equivalent to the canonical basis for $l^{p}$ and $\displaystyle\lim_{k\to\infty}\|f_{n_{k}}\|_{2}=0$; or
\item $(f_{n_{k}})_{k\geq 1}$ is equivalent to the canonical basis for $l^{2}$ and $\displaystyle\inf_{k\geq 1}\|f_{n_{k}}\|_{2}>0$.
\end{enumerate}
\end{lemma}
\begin{lemma}\label{2cont}
Let $2<p<\infty$. Let $T\in B(L^{p})$. Then for every $\epsilon>0$, there exists $\delta>0$ such that $\|Tf\|_{2}<\epsilon$ for every $f\in L^{p}$ with $\|f\|_{p}\leq 1$ and $\|f\|_{2}<\delta$.
\end{lemma}
\begin{proof}
Suppose by contradiction there are $f_{1},f_{2},\ldots\in L^{p}$ and $\epsilon>0$ such that $\|f_{n}\|_{p}\leq 1$, $\|f_{n}\|_{2}\to 0$ and $\|Tf_{n}\|_{2}\geq\epsilon$. Then $f_{n}\to 0$ weakly in $L^{p}$. By Lemma \ref{Kad} on $(Tf_{n})_{n\geq 1}$ and $(f_{n})_{n\geq 1}$, passing to a subsequence, we have that $(Tf_{n})_{n\geq 1}$ is equivalent to the canonical basis for $l^{2}$ and that either $\|f_{n}\|_{p}\to 0$ or $(f_{n})_{n\geq 1}$ is equivalent to the canonical basis for $l^{p}$. But this is an absurdity since $T$ is bounded and $p>2$.
\end{proof}
\begin{lemma}\label{minvariant}
Let $2<p<\infty$. The space $\mathcal{M}$ defined in (\ref{m}) is invariant under $\rho_{\mathcal{U}}(\pi(T))$ for all $T\in B(L^{p})$.
\end{lemma}
\begin{proof}
By Lemma \ref{2cont}, $T^{\mathcal{U}}\mathcal{M}\subset\mathcal{M}$ for all $T\in B(L^{p})$. Since $\mathcal{M}\subset\widehat{L^{p}}$ and $\rho_{\mathcal{U}}(\pi(T))=\widehat{T}$ is the restriction of $T^{\mathcal{U}}$ to $\widehat{L^{p}}$, it follows that $\rho_{\mathcal{U}}(\pi(T))\mathcal{M}\subset\mathcal{M}$ for all $T\in B(L^{p})$.
\end{proof}
Lemma \ref{minvariant} proves one direction of Theorem \ref{main}, namely,
\[\{\rho_{\mathcal{U}}(\pi(T)):T\in B(L^{p})\}^{-WOT}\subset\{S\in B(\widehat{L^{p}}):S\mathcal{M}\subset\mathcal{M}\}.\]
This holds for all nonprincipal ultrafilter $\mathcal{U}$. But to prove the other direction, we need to assume that $\mathcal{U}$ is selective.
\section{Projection onto $\mathcal{M}$}
In this section, we construct a projection onto $\mathcal{M}$. This is needed in the proof of Theorem \ref{main}.
\begin{lemma}\label{limitexists}
Let $1<p<\infty$. Let $(f_{n})_{n\geq 1}$ be a sequence in $L^{p}$ satisfying $\displaystyle\sup_{n\geq 1}\|f_{n}\|_{p}<\infty$. Then
\begin{enumerate}
\item $(f_{n}I(|f_{n}|>r))_{n,\mathcal{U}}$ converges in norm to an element in $(L^{p})^{\mathcal{U}}$ as $r\to\infty$; and
\item \[\lim_{r\to\infty}(rI(|f_{n}|>r))_{n,\mathcal{U}}=0.\]
\end{enumerate}
\end{lemma}
\begin{proof}
Recall that $L^{p}=L^{p}([0,1],\mu)$ where $\mu$ is the Lebesgue measure on $[0,1]$. Let $\nu_{n}$ be the pushforward probability measure on $\mathbb{C}$ under $f_{n}$ of $\mu$. Since $\displaystyle\sup_{n\geq 1}\|f_{n}\|_{p}<\infty$, the measures $\nu_{n}$ are uniformly tight, i.e., for every $\epsilon>0$, there exists a compact set $\mathcal{K}$ in $\mathbb{C}$ such that $\nu_{n}(\mathcal{K})\geq1-\epsilon$ for all $n\geq 1$. Recall that every uniformly tight sequence of probability measures on $\mathbb{C}$ has a subsequence that converges weakly to a probability measure on $\mathbb{C}$. Since the measures $\nu_{n}$ are uniformly tight, there exists a weak limit $\nu$ of $(\nu_{n})_{n\geq 1}$ through $\mathcal{U}$. Note that
\begin{equation}\label{pint}
\int|z|^{p}\,d\nu(z)\leq\lim_{n,\mathcal{U}}\int|z|^{p}\,d\nu_{n}(z)=\lim_{n,\mathcal{U}}\int|f_{n}|^{p}\,d\mu<\infty.
\end{equation}
So for every $r_{2}>r_{1}>0$,
\begin{align*}
&\|(f_{n}I(|f_{n}|>r_{1}))_{n,\mathcal{U}}-(f_{n}I(|f_{n}|>r_{2}))_{n,\mathcal{U}}\|^{p}\\=&
\lim_{n,\mathcal{U}}\|f_{n}I(r_{1}<|f_{n}|\leq r_{2})\|_{p}^{p}\\=&
\lim_{n,\mathcal{U}}\int_{r_{1}<|z|\leq r_{2}}|z|^{p}\,d\nu_{n}(z)\\\leq&\int_{r_{1}\leq|z|\leq r_{2}}|z|^{p}\,d\nu(z).
\end{align*}
Therefore,
\[\lim_{r_{1},r_{2}\to\infty}\|(f_{n}I(f_{n}>r_{1}))_{n,\mathcal{U}}-(f_{n}I(f_{n}>r_{2}))_{n,\mathcal{U}}\|=0.\]
So $(f_{n}I(|f_{n}|>r))_{n,\mathcal{U}}$ converges in norm to an element in $(L^{p})^{\mathcal{U}}$ as $r\to\infty$. This proves the first assertion.

For the second assertion, observe that
\[\lim_{n,\mathcal{U}}\mu(|f_{n}|>r)=\lim_{n,\mathcal{U}}\nu_{n}(\{z\in\mathbb{C}:|z|>r\})\leq\nu(\{z\in\mathbb{C}:|z|\geq r\}).\]
So
\begin{align*}
&\limsup_{r\to\infty}\|(rI(|f_{n}|>r)_{n,\mathcal{U}}\|^{p}\\=&\limsup_{r\to\infty}r^{p}\lim_{n,\mathcal{U}}\mu(|f_{n}|>r)\\\leq&
\limsup_{r\to\infty}r^{p}\nu(\{z\in\mathbb{C}:|z|\geq r\})\\\leq&\limsup_{r\to\infty}\int_{|z|\geq r}|z|^{p}\,d\nu(z)=0,
\end{align*}
where the last equality follows from (\ref{pint}).
\end{proof}
\begin{lemma}\label{linear}
Let $1<p<\infty$. Let $(f_{n})_{n\geq 1}$ and $(g_{n})_{n\geq 1}$ be sequences of nonnegative functions in $L^{p}$ such that $\displaystyle\sup_{n\geq 1}\|f_{n}\|_{p}<\infty$ and $\displaystyle\sup_{n\geq 1}\|g_{n}\|_{p}<\infty$. Then
\begin{align*}
&\lim_{r\to\infty}((f_{n}+g_{n})I(f_{n}+g_{n}>r))_{n,\mathcal{U}}\\=&
\lim_{r\to\infty}(f_{n}I(f_{n}>r))_{n,\mathcal{U}}+\lim_{r\to\infty}(g_{n}I(g_{n}>r))_{n,\mathcal{U}}.
\end{align*}
\end{lemma}
\begin{proof}
For every $r>0$,
\begin{equation}\label{superadditive}
(f_{n}+g_{n})I(f_{n}+g_{n}>r)\geq f_{n}I(f_{n}>r)+g_{n}I(g_{n}>r).
\end{equation}
Let $0<s<r$. Then
\[I(f_{n}+g_{n}>r)\leq I(f_{n}>s)+I(f_{n}\leq s\text{ and }g_{n}>r-s).\]
So
\[f_{n}I(f_{n}+g_{n}>r)\leq f_{n}I(f_{n}>s)+sI(g_{n}>r-s).\]
Interchanging the roles of $f_{n}$ and $g_{n}$, we obtain
\[g_{n}I(f_{n}+g_{n}>r)\leq g_{n}I(g_{n}>s)+sI(f_{n}>r-s).\]
Therefore,
\begin{align*}
&(f_{n}+g_{n})I(f_{n}+g_{n}>r)\\\leq&f_{n}I(f_{n}>s)+g_{n}I(g_{n}>s)+s(I(f_{n}>r-s)+I(g_{n}>r-s)).
\end{align*}
Combining this with (\ref{superadditive}), we have
\begin{align*}
0\leq&(f_{n}+g_{n})I(f_{n}+g_{n}>r)-f_{n}I(f_{n}>r)-g_{n}I(g_{n}>r)\\\leq&f_{n}I(f_{n}>s)-f_{n}I(f_{n}>r)+g_{n}I(g_{n}>s)-g_{n}I(g_{n}>r)\\&
+s(I(f_{n}>r-s)+I(g_{n}>r-s)).
\end{align*}
Thus,
\begin{align}\label{additivebound}
&\lim_{n,\mathcal{U}}\|(f_{n}+g_{n})I(f_{n}+g_{n}>r)-f_{n}I(f_{n}>r)-g_{n}I(g_{n}>r)\|\\\leq&
\lim_{n,\mathcal{U}}\|f_{n}I(f_{n}>s)-f_{n}I(f_{n}>r)\|+\lim_{n,\mathcal{U}}\|g_{n}I(g_{n}>s)-g_{n}I(g_{n}>r)\|\nonumber\\&+
\lim_{n,\mathcal{U}}\|s(I(f_{n}>r-s)+I(g_{n}>r-s))\|.\nonumber
\end{align}
Since $\displaystyle\sup_{n\geq 1}\|f_{n}\|_{p}<\infty$ and $\displaystyle\sup_{n\geq 1}\|g_{n}\|_{p}<\infty$, by Markov's inequality,
\[\lim_{r\to\infty}\|(I(f_{n}>r)+I(g_{n}>r))_{n,\mathcal{U}}\|=0.\]
Thus,
\[\lim_{r\to\infty}\lim_{n,\mathcal{U}}\|s(I(f_{n}>r-s)+I(g_{n}>r-s))\|=0,\]
for every $s>0$. So by (\ref{additivebound}),
\begin{align*}
&\limsup_{r\to\infty}\|((f_{n}+g_{n})I(f_{n}+g_{n}>r))_{n,\mathcal{U}}\\
&-(f_{n}I(f_{n}>r))_{n,\mathcal{U}}-(g_{n}I(g_{n}>r))_{n,\mathcal{U}}\|\\\leq&
\limsup_{r\to\infty}\|(f_{n}I(f_{n}>s))_{n,\mathcal{U}}-(f_{n}I(f_{n}>r))_{n,\mathcal{U}}\|\\&+\limsup_{r\to\infty}\|(g_{n}I(g_{n}>s))_{n,\mathcal{U}}-
(g_{n}I(g_{n}>r))_{n,\mathcal{U}}\|,
\end{align*}
for every $s>0$. By Lemma \ref{limitexists}, taking $s\to\infty$, we have
\begin{align*}
&\limsup_{r\to\infty}\|((f_{n}+g_{n})I(f_{n}+g_{n}>r))_{n,\mathcal{U}}\\&-(f_{n}I(f_{n}>r))_{n,\mathcal{U}}-(g_{n}I(g_{n}>r))_{n,\mathcal{U}}\|=0.
\end{align*}
Thus the result follows.
\end{proof}
The following lemma is elementary but we include its proof for convenience.
\begin{lemma}\label{linearcriteria}
Let $V$ and $W$ be vector spaces over $\mathbb{R}$. Let $V^{+}$ be a subset of $V$ such that $x_{1}+x_{2}\in V^{+}$ for every $x_{1},x_{2}\in V^{+}$. Let $+,-:V\to V^{+}$ be functions such that $x=x^{+}-x^{-}$ for every $x\in V$. Let $\phi:V\to W$ be a map such that
\[\phi(ax)=a\phi(x),\quad x\in V,\,a\in\mathbb{R},\]
\[\phi(x)=\phi(x^{+})-\phi(x^{-}),\quad x\in V,\]
and
\[\phi(x_{1}+x_{2})=\phi(x_{1})+\phi(x_{2}),\quad x_{1},x_{2}\in V^{+}.\]
Then $\phi$ is linear.
\end{lemma}
\begin{proof}
Let $x_{1},x_{2}\in V^{+}$. Then
\[x_{1}-x_{2}=(x_{1}-x_{2})^{+}-(x_{1}-x_{2})^{-}.\]
$\Longrightarrow$
\[(x_{1}-x_{2})^{+}+x_{2}=(x_{1}-x_{2})^{-}+x_{1}.\]
$\Longrightarrow$
\[\phi((x_{1}-x_{2})^{+})+\phi(x_{2})=\phi((x_{1}-x_{2})^{-})+\phi(x_{1}).\]
$\Longrightarrow$
\[\phi((x_{1}-x_{2})^{+})-\phi((x_{1}-x_{2})^{-})=\phi(x_{1})-\phi(x_{2}).\]
$\Longrightarrow$
\[\phi(x_{1}-x_{2})=\phi(x_{1})-\phi(x_{2}).\]
Therefore,
\begin{equation}\label{minus}
\phi(x_{1}-x_{2})=\phi(x_{1})-\phi(x_{2}),\quad x_{1},x_{2}\in V^{+}.
\end{equation}
For $y_{1},y_{2}\in V$,
\[y_{1}+y_{2}=(y_{1}^{+}+y_{2}^{+})-(y_{1}^{-}+y_{2}^{-}).\]
Since $y_{1}^{+}+y_{2}^{+}$ and $y_{1}^{-}+y_{2}^{-}$ are in $V^{+}$, by (\ref{minus}),
\begin{align*}
&\phi(y_{1}+y_{2})\\=&\phi(y_{1}^{+}+y_{2}^{+})-\phi(y_{1}^{-}+y_{2}^{-})\\=&\phi(y_{1}^{+})+\phi(y_{2}^{+})-\phi(y_{1}^{-})-\phi(y_{2}^{-})\\=&
\phi(y_{1}^{+})-\phi(y_{1}^{-})+\phi(y_{2}^{+})-\phi(y_{2}^{-})=\phi(y_{1})+\phi(y_{2}).
\end{align*}
Thus, the result follows.
\end{proof}
\begin{lemma}\label{rdefined}
Let $1<p<\infty$. The map $R_{p}:(L^{p})^{\mathcal{U}}\to(L^{p})^{\mathcal{U}}$,
\[R_{p}[(f_{n})_{n,\mathcal{U}}]=\lim_{r\to\infty}(f_{n}I(|f_{n}|>r))_{n,\mathcal{U}},\quad (f_{n})_{n,\mathcal{U}}\in(L^{p})^{\mathcal{U}}\]
is a well defined operator on $(L^{p})^{\mathcal{U}}$.
\end{lemma}
\begin{proof}
We first treat the case when the scalar field is $\mathbb{R}$. Define $S:l^{\infty}(L^{p})\to(L^{p})^{\mathcal{U}}$ by
\[S[(f_{n})_{n\geq 1}]=\lim_{r\to\infty}(f_{n}I(|f_{n}|>r))_{n,\mathcal{U}}.\]
Take $V=l^{\infty}(L^{p})$, $W=(L^{p})^{\mathcal{U}}$, $V^{+}=\{(f_{n})_{n\geq 1}\in l^{\infty}(L^{p}):f_{n}\geq 0\text{ for all }n\geq 1\}$, $[(f_{n})_{n\geq 1}]^{+}=(f_{n}^{+})_{n\geq 1}$ and $[(f_{n})_{n\geq 1}]^{-}=(f_{n}^{-})_{n\geq 1}$, where $f_{n}^{+}$ and $f_{n}^{-}$ are the positive and negative parts of $f_{n}$, respectively. By Lemmas \ref{linear} and \ref{linearcriteria}, it follows that $S$ is linear. Obviously $S$ is bounded.

Observe that if $\displaystyle\lim_{n,\mathcal{U}}\|f_{n}\|_{p}=0$ then $S[(f_{n})_{n\geq 1}]=0$. Thus $R_{p}$ is a well defined operator on $(L^{p})^{\mathcal{U}}$.

We now treat the case when the scalar field is $\mathbb{C}$. Since $R_{p}$ is a well defined operator when the scalar field is $\mathbb{R}$, it suffices to prove that if we write $f_{n}=f_{n}^{(1)}+if_{n}^{(2)}$ where $f_{n}^{(1)},f_{n}^{(2)}$ take real values, then
\begin{align}\label{compdecomp}
&\lim_{r\to\infty}(f_{n}I(|f_{n}|>r))_{n,\mathcal{U}}\nonumber\\=&
\lim_{r\to\infty}(f_{n}^{(1)}I(|f_{n}^{(1)}|>r))_{n,\mathcal{U}}+i\lim_{r\to\infty}(f_{n}^{(2)}I(|f_{n}^{(2)}|>r))_{n,\mathcal{U}}.
\end{align}
Observe that
\begin{align*}
&\lim_{r\to\infty}\|(f_{n}^{(1)}I(|f_{n}|>r))_{n,\mathcal{U}}-(f_{n}^{(1)}I(|f_{n}^{(1)}|>r))_{n,\mathcal{U}}\|\\=&
\lim_{r\to\infty}\|(|f_{n}^{(1)}|I(|f_{n}|>r\text{ and }|f_{n}^{(1)}|\leq r))_{n,\mathcal{U}}\|\leq\lim_{r\to\infty}\|(rI(|f_{n}|>r))_{n,\mathcal{U}}\|=0,
\end{align*}
where the last equality follows from Lemma \ref{limitexists}. Thus,
\[\lim_{r\to\infty}(f_{n}^{(1)}I(|f_{n}|>r))_{n,\mathcal{U}}=\lim_{r\to\infty}(f_{n}^{(1)}I(|f_{n}^{(1)}|>r))_{n,\mathcal{U}}.\]
Similarly we have
\[\lim_{r\to\infty}(f_{n}^{(2)}I(|f_{n}|>r))_{n,\mathcal{U}}=\lim_{r\to\infty}(f_{n}^{(2)}I(|f_{n}^{(2)}|>r))_{n,\mathcal{U}}.\]
So (\ref{compdecomp}) is proved.
\end{proof}
Recall that $\mathcal{M}$ is defined in (\ref{m}).
\begin{lemma}\label{mproj}
Let $2<p<\infty$. The map $R_{p}:(L^{p})^{\mathcal{U}}\to(L^{p})^{\mathcal{U}}$ defined by
\[R_{p}[(f_{n})_{n,\mathcal{U}}]=\lim_{r\to\infty}(f_{n}I(|f_{n}|>r))_{n,\mathcal{U}},\quad (f_{n})_{n,\mathcal{U}}\in(L^{p})^{\mathcal{U}}\]
is a norm $1$ projection onto $\mathcal{M}$.
\end{lemma}
\begin{proof}
In view of Lemma \ref{rdefined}, it suffices to show that
\begin{equation}\label{rinm}
R_{p}[(f_{n})_{n,\mathcal{U}}]\in\mathcal{M},\quad (f_{n})_{n,\mathcal{U}}\in(L^{p})^{\mathcal{U}},
\end{equation}
and
\begin{equation}\label{minr}
R_{p}[(f_{n})_{n,\mathcal{U}}]=(f_{n})_{n,\mathcal{U}},\quad (f_{n})_{n,\mathcal{U}}\in\mathcal{M}.
\end{equation}
Let $(f_{n})_{n,\mathcal{U}}\in(L^{p})^{\mathcal{U}}$. Let $(g_{n})_{n,\mathcal{U}}=R_{p}[(f_{n})_{n,\mathcal{U}}]$. Then
\[\lim_{r\to\infty}\|(f_{n}I(|f_{n}|>r))_{n,\mathcal{U}}-(g_{n})_{n,\mathcal{U}}\|=0,\]
and so since $p>2$,
\begin{equation}\label{2approx}
\lim_{r\to\infty}\lim_{n,\mathcal{U}}\|f_{n}I(|f_{n}|>r)-g_{n}\|_{2}\leq\lim_{r\to\infty}\lim_{n,\mathcal{U}}\|f_{n}I(|f_{n}|>r)-g_{n}\|_{p}=0.
\end{equation}
Note that
\[\|f_{n}I(|f_{n}|>r)\|_{2}^{2}\leq\frac{1}{r^{p-2}}\|f_{n}\|_{p}^{p}.\]
So
\[\lim_{r\to\infty}\lim_{n,\mathcal{U}}\|f_{n}I(|f_{n}|>r)\|_{2}=0.\]
Thus by (\ref{2approx}), $\displaystyle\lim_{n,\mathcal{U}}\|g_{n}\|_{2}=0$ and so $R_{p}[(f_{n})_{n,\mathcal{U}}]\in\mathcal{M}$. This proves (\ref{rinm}).

To prove (\ref{minr}), note that
\[\|(f_{n})_{n,\mathcal{U}}-R_{p}[(f_{n})_{n,\mathcal{U}}]\|=\lim_{r\to\infty}\lim_{n,\mathcal{U}}\|f_{n}I(|f_{n}|\leq r)\|_{p}.\]
But $\|f_{n}I(|f_{n}|\leq r)\|_{p}^{p}\leq r^{p-2}\|f_{n}\|_{2}^{2}$. Therefore, $R_{p}[(f_{n})_{n,\mathcal{U}}]=(f_{n})_{n,\mathcal{U}}$ for all $(f_{n})_{n,\mathcal{U}}\in\mathcal{M}$.
\end{proof}
\section{Properties of $R_{p}$}
In this section, we prove some properties of the projection $R_{p}$ defined in Lemma \ref{mproj}. These properties are needed in the proof of Theorem \ref{main}.

Let $1<p<\infty$. Let $\displaystyle\frac{1}{p}+\frac{1}{q}=1$. Recall from the preliminaries in Section 1 that the dual of $(L^{p})^{\mathcal{U}}$ can be identified with $(L^{q})^{\mathcal{U}}$ via the following duality relation
\[((f_{n})_{n,\mathcal{U}},(g_{n})_{n,\mathcal{U}})=\lim_{n,\mathcal{U}}\int f_{n}g_{n}\,d\mu,\]
where $(f_{n})_{n,\mathcal{U}}\in(L^{p})^{\mathcal{U}}$ and $(g_{n})_{n,\mathcal{U}}\in(L^{q})^{\mathcal{U}}$.
\begin{lemma}\label{sa}
If $1<p<\infty$ and $\displaystyle\frac{1}{p}+\frac{1}{q}=1$, then the adjoint of $R_{p}$ is $R_{q}$.
\end{lemma}
\begin{proof}
Let $(f_{n})_{n,\mathcal{U}}\in(L^{p})^{\mathcal{U}}$ and $(g_{n})_{n,\mathcal{U}}\in(L^{q})^{\mathcal{U}}$. We need to show that
\[(R_{p}[(f_{n})_{n,\mathcal{U}}],(g_{n})_{n,\mathcal{U}})=((f_{n})_{n,\mathcal{U}},R_{q}[(g_{n})_{n,\mathcal{U}}]).\]
It is obvious that
\begin{align}\label{adjoint}
&(R_{p}[(f_{n})_{n,\mathcal{U}}],(g_{n})_{n,\mathcal{U}})-((f_{n})_{n,\mathcal{U}},R_{q}[(g_{n})_{n,\mathcal{U}}])\\=&
(R_{p}[(f_{n})_{n,\mathcal{U}}],(g_{n})_{n,\mathcal{U}}-R_{q}[(g_{n})_{n,\mathcal{U}}])-
((f_{n})_{n,\mathcal{U}}-R_{p}[(f_{n})_{n,\mathcal{U}}],R_{q}[(g_{n})_{n,\mathcal{U}}]).\nonumber
\end{align}
Note that
\begin{align*}
&(R_{p}[(f_{n})_{n,\mathcal{U}}],(g_{n})_{n,\mathcal{U}}-R_{q}[(g_{n})_{n,\mathcal{U}}])\\=&
\lim_{s\to\infty}(R_{p}[(f_{n})_{n,\mathcal{U}}],(g_{n}I(|g_{n}|\leq s))_{n,\mathcal{U}})\\=&
\lim_{s\to\infty}\lim_{r\to\infty}((f_{n}I(|f_{n}|>r))_{n,\mathcal{U}},(g_{n}I(|g_{n}|\leq s))_{n,\mathcal{U}})\\=&
\lim_{s\to\infty}\lim_{r\to\infty}\lim_{n,\mathcal{U}}\int f_{n}g_{n}I(|f_{n}|>r)I(|g_{n}|\leq s)\,d\mu.
\end{align*}
So
\begin{align*}
&|(R_{p}[(f_{n})_{n,\mathcal{U}}],(g_{n})_{n,\mathcal{U}}-R_{q}[(g_{n})_{n,\mathcal{U}}])|\\\leq&
\limsup_{s\to\infty}\limsup_{r\to\infty}\lim_{n,\mathcal{U}}\int|f_{n}||g_{n}|I(|f_{n}|>r)I(|g_{n}|\leq s)\,d\mu\\\leq&
\limsup_{s\to\infty}\limsup_{r\to\infty}\lim_{n,\mathcal{U}}s\int|f_{n}|I(|f_{n}|>r)\,d\mu\\\leq&
\limsup_{s\to\infty}\limsup_{r\to\infty}\lim_{n,\mathcal{U}}s\frac{1}{r^{p-1}}\int|f_{n}|^{p}\,d\mu\\=&
\limsup_{s\to\infty}\limsup_{r\to\infty}s\frac{1}{r^{p-1}}\|(f_{n})_{n,\mathcal{U}}\|^{p}=0.
\end{align*}
Hence
\[(R_{p}[(f_{n})_{n,\mathcal{U}}],(g_{n})_{n,\mathcal{U}}-R_{q}[(g_{n})_{n,\mathcal{U}}])=0.\]
Interchanging the roles of $(f_{n})_{n,\mathcal{U}}$ and $(g_{n})_{n,\mathcal{U}}$ and the roles of $p$ and $q$, we have
\[((f_{n})_{n,\mathcal{U}}-R_{p}[(f_{n})_{n,\mathcal{U}}],R_{q}[(g_{n})_{n,\mathcal{U}}])=0.\]
Thus by (\ref{adjoint}) the result follows.
\end{proof}
Recall that $\widehat{L^{p}}$ is defined in Section 1.
\begin{lemma}\label{rhat}
Let $1<p<\infty$. Then the range of $R_{p}$ is contained in $\widehat{L^{p}}$.
\end{lemma}
\begin{proof}
Let $(f_{n})_{n,\mathcal{U}}\in(L^{p})^{\mathcal{U}}$. Let $g\in L^{q}$ where $\displaystyle\frac{1}{p}+\frac{1}{q}=1$. By definition of $R_{q}$, we have $R_{q}[(g)_{n,\mathcal{U}}]=0$. So by Lemma \ref{sa},
\[(R_{p}[(f_{n})_{n,\mathcal{U}}],(g)_{n,\mathcal{U}})=((f_{n})_{n,\mathcal{U}},R_{q}[(g)_{n,\mathcal{U}}])=0.\]
So $R_{p}[(f_{n})_{n,\mathcal{U}}]\in\widehat{L^{p}}$.
\end{proof}
\section{Uses of selectivity}
In this section, we prove some lemmas that assume, in an essential way, that the ultrafilter is selective. These lemmas are needed in the proof of Theorem \ref{main}. Recall that the definition of selective is given in Section 1.
\begin{lemma}\label{selconv}
Let $\mathcal{U}$ be a selective nonprincipal ultrafilter on $\mathbb{N}$. Let $Z$ be a metric space. Let $(x_{n})_{n\geq 1}$ be a sequence in $Z$ converging to an element $x\in Z$ through $\mathcal{U}$. Then there exists $A\in\mathcal{U}$ such that the subsequence $(x_{n})_{n\in A}$ converges to $x$.
\end{lemma}
\begin{proof}
For each $k\geq 1$, let $\displaystyle A_{k}=\{n\in\mathbb{N}:d(x_{n},x)<\frac{1}{k}\}$. Then $A_{k}\in\mathcal{U}$ for all $k\geq 1$. Since $\mathcal{U}$ is selective, there exists $A\in\mathcal{U}$ that is almost contained in each $A_{k}$. So the subsequence $(x_{n})_{n\in A}$ converges to $x$.
\end{proof}
\begin{lemma}\label{selectivesubseq}
Let $\mathcal{U}$ be a selective nonprincipal ultrafilter on $\mathbb{N}$. Let $A_{1},A_{2},\ldots$ be sets in $\mathcal{U}$. Then there exists a sequence $(k_{n})_{n\geq 1}$ in $\mathbb{N}$ such that $k_{n}\to\infty$ as $n\to\infty$ and
\[\{n\in\mathbb{N}:n\in A_{k_{n}}\}\in\mathcal{U}.\]
\end{lemma}
\begin{proof}
Since $\mathcal{U}$ is selective, there exists $A\in\mathcal{U}$ that is almost contained in each $A_{k}$. We may assume that $A\subset A_{1}$. For each $n\geq 1$, let
\[k_{n}=\begin{array}{cc}\begin{cases}\sup\{k\geq 1:n\in A_{1}\cap\ldots\cap A_{k}\},&n\in A\backslash(A_{1}\cap A_{2}\cap\ldots)\\n,&n\in(A_{1}\cap A_{2}\cap\ldots)\cup A^{c}\end{cases}.\end{array}\]
Observe that $A\subset\{n\in\mathbb{N}:n\in A_{k_{n}}\}$. Since $A\in\mathcal{U}$, it follows that $\{n\in\mathbb{N}:n\in A_{k_{n}}\}\in\mathcal{U}$.

It remains to show that $k_{n}\to\infty$. Since $A$ is almost contained in each $A_{k}$, we have that $A$ is almost contained in $A_{1}\cap\ldots\cap A_{k}$ for each $k\geq 1$. Thus, for each $k\geq 1$, there exists $N(k)\geq 1$ such that $A\cap[N(k),\infty)\subset A_{1}\cap\ldots\cap A_{k}$. So for every $k\geq 1$ and $n\in[N(k),\infty)\cap A\backslash(A_{1}\cap A_{2}\cap\ldots)$, we have $k_{n}\geq k$. So for every $k\geq 1$ and $n\geq N(k)$, we have $k_{n}\geq\min(k,n)$. So $k_{n}\to\infty$ as $n\to\infty$.
\end{proof}
\begin{lemma}\label{disjoint}
Let $1<p<\infty$. Let $\mathcal{U}$ be a selective nonprincipal ultrafilter on $\mathbb{N}$. If $(f_{n})_{n,\mathcal{U}}\in\mathrm{ran}\,R_{p}$ then there are $g_{1},g_{2},\ldots\in L^{p}$ with disjoint supports such that $(f_{n})_{n,\mathcal{U}}=(g_{n})_{n,\mathcal{U}}$.
\end{lemma}
\begin{proof}
Since $(f_{n})_{n,\mathcal{U}}\in \mathrm{ran}\,R_{p}$ and $R_{p}$ is a projection,
\[\lim_{r\to\infty}(f_{n}I(|f_{n}|\leq r))_{n,\mathcal{U}}=0.\]
Thus for every $k\geq 1$,
\[\lim_{n,\mathcal{U}}\|f_{n}I(|f_{n}|\leq k)\|_{p}\leq\lim_{r\to\infty}\lim_{n,\mathcal{U}}\|f_{n}I(|f_{n}|\leq r)\|_{p}=0.\]
For every $k\geq 1$, let $\displaystyle A_{k}=\{n\geq 1:\|f_{n}I(|f_{n}|\leq k)\|_{p}\leq\frac{1}{k}\}$. Since $\mathcal{U}$ is selective, by Lemma \ref{selectivesubseq}, there exists a sequence $(k_{n})_{n\geq 1}$ in $\mathbb{N}$ such that $k_{n}\to\infty$ and
\[\{n\geq 1:\|f_{n}I(|f_{n}|\leq k_{n})\|_{p}\leq\frac{1}{k_{n}}\}\in\mathcal{U}.\]
So
\[(f_{n})_{n,\mathcal{U}}=(f_{n}I(|f_{n}|>k_{n}))_{n,\mathcal{U}}.\]
Since $\displaystyle\sup_{n\geq 1}\|f_{n}\|_{p}<\infty$, by Markov's inequality, $\mu(|f_{n}|>k_{n})\to 0$. Therefore without loss of generality, we may assume that $\mu(\mathrm{supp}(f_{n}))\to 0$.

Choose $0=m(0)<m(1)<m(2)<\ldots$ as follows:\\
Since $\|f_{1}I(\mathrm{supp}(f_{n}))\|_{p}\to 0$ as $n\to\infty$, there exists $m(1)\geq 1$ such that
\[\|f_{1}I(\mathrm{supp}(f_{n}))\|_{p}\leq\frac{1}{2},\quad n\geq m(1).\]
There exists $m(2)>m(1)$ such that
\[\|f_{m}I(\mathrm{supp}(f_{n}))\|_{p}\leq\frac{1}{2^{2}},\quad m\leq m(1),\,n\geq m(2).\]
Suppose that $m(1),\ldots,m(k-1)$ have been chosen. There exists $m(k)>m(k-1)$ such that
\begin{equation}\label{weaknull}
\|f_{m}I(\mathrm{supp}(f_{n}))\|_{p}\leq\frac{1}{2^{k}},\quad m\leq m(k-1),\,n\geq m(k).
\end{equation}
Note that $\{[m(k)+1,m(k+1)]:k\geq 0\}$ is a partition of $\mathbb{N}$. So
\[\mathbb{N}=\left(\bigcup_{k\text{ even}}[m(k)+1,m(k+1)]\right)\cup\left(\bigcup_{k\text{ odd}}[m(k)+1,m(k+1)]\right).\]
Since $\mathcal{U}$ is an ultrafilter, it contains exactly one of these two sets. For simplicity, assume that it contains the first one $\cup_{k\text{ even}}[m(k)+1,m(k+1)]$. Since $\mathcal{U}$ is selective, there exists $B\in\mathcal{U}$ such that $B\cap[m(2k)+1,m(2k+1)]$ is a singleton for each $k\geq 0$ and $B\subset\cup_{k\text{ even}}[m(k)+1,m(k+1)]$. Write $B=\{t(0),t(1),\ldots\}$ where $t(0)<t(1)<\ldots$. We have $m(2k)+1\leq t(k)\leq m(2k+1)$ so by (\ref{weaknull}),
\[\|f_{t(j)}I(\mathrm{supp}(f_{t(k)}))\|_{p}\leq\frac{1}{2^{2k}},\quad 0\leq j<k.\]
Thus
\[\|f_{t(j)}I(\bigcup_{k=j+1}^{\infty}\mathrm{supp}(f_{t(k)}))\|_{p}\leq\frac{1}{2^{2j}},\quad j\geq 0.\]
Let
\[g_{t(j)}=f_{t(j)}I(\bigcup_{k=j+1}^{\infty}\mathrm{supp}(f_{t(k)}))^{c},\quad j\geq 0,\]
and
\[g_{n}=0,\quad n\notin B.\]
Then $(f_{n})_{n,\mathcal{U}}=(g_{n})_{n,\mathcal{U}}$ and all the $g_{n}$ have disjoint supports. Thus the result follows.
\end{proof}
\begin{lemma}\label{xy}
Let $\mathcal{U}$ be a selective nonprincipal ultrafilter on $\mathbb{N}$. Let $\mathcal{X}$ be a reflexive Banach space. Let $(y_{n})_{n,\mathcal{U}}\in\widehat{\mathcal{X}}$ and $(x_{n}^{*})_{n,\mathcal{U}}\in\widehat{\mathcal{X}^{*}}$.
Then there exists $B=\{t(0),t(1),\ldots\}\in\mathcal{U}$, where $t(0)<t(1)<\ldots$ such that
\[|x_{t(j)}^{*}(y_{t(k)})|\leq\frac{1}{2^{\max(j,k)}},\quad j\neq k.\]
\end{lemma}
\begin{proof}
Since $\displaystyle w\hyphen\lim_{n,\mathcal{U}}y_{n}=0$ and $\displaystyle w\hyphen\lim_{n,\mathcal{U}}x_{n}^{*}=0$, by Lemma \ref{selconv}, we may assume that $y_{n}\to 0$ and $x_{n}^{*}\to 0$ weakly as $n\to\infty$.

Choose $0=m(0)<m(1)<m(2)<\ldots$ as follows:\\
Since $x_{1}^{*}(y_{n})\to 0$ and $x_{n}^{*}(y_{1})\to 0$, there exists $m(1)\geq 1$ such that
\[|x_{1}^{*}(y_{n})|\leq\frac{1}{2}\text{ and }|x_{n}^{*}(y_{1})|\leq\frac{1}{2},\quad n\geq m(1).\]
There exists $m(2)>m(1)$ such that
\[|x_{m}^{*}(y_{n})|\leq\frac{1}{2^{2}}\text{ and }|x_{n}^{*}(y_{m})|\leq\frac{1}{2^{2}},\quad m\leq m(1),\,n\geq m(2).\]
Suppose that $m(1),\ldots,m(k-1)$ have been chosen. There exists $m(k)>m(k-1)$ such that
\begin{equation}\label{biorth}
|x_{m}^{*}(y_{n})|\leq\frac{1}{2^{k}}\text{ and }|x_{n}^{*}(y_{m})|\leq\frac{1}{2^{k}},\quad m\leq m(k-1),\,n\geq m(k).
\end{equation}
Since $\{[m(k)+1,m(k+1)]:k\geq 0\}$ is a partition of $\mathbb{N}$,
\[\mathbb{N}=\left(\bigcup_{k\text{ even}}[m(k)+1,m(k+1)]\right)\cup\left(\bigcup_{k\text{ odd}}[m(k)+1,m(k+1)]\right).\]
Since $\mathcal{U}$ is an ultrafilter, it contains exactly one of these two sets, say, the first one $\cup_{k\text{ even}}[m(k)+1,m(k+1)]$. Since $\mathcal{U}$ is selective, there exists $B\in\mathcal{U}$ such that $B\cap[m(2k)+1,m(2k+1)]$ is a singleton for each $k\geq 0$ and $B\subset\cup_{k\text{ even}}[m(k)+1,m(k+1)]$. Write $B=\{t(0),t(1),\ldots\}$ where $t(0)<t(1)<\ldots$. We have $m(2k)+1\leq t(k)\leq m(2k+1)$ so by (\ref{biorth}),
$$|x_{t(j)}^{*}(y_{t(k)})|\leq\frac{1}{2^{2k}}\text{ and }|x_{t(k)}^{*}(y_{t(j)})|\leq\frac{1}{2^{2k}},\quad 0\leq j<k.\eqno\qedhere$$
\end{proof}
\begin{lemma}\label{blockhaar}
Let $1<p<\infty$. Let $\mathcal{U}$ be a selective nonprincipal ultrafilter on $\mathbb{N}$. If $(f_{n})_{n,\mathcal{U}}\in\widehat{L^{p}}$ then there are $g_{1},g_{2},\ldots\in L^{p}$ such that $(f_{n})_{n,\mathcal{U}}=(g_{n})_{n,\mathcal{U}}$ and $(g_{n})_{n\geq 1}$ is a block sequence of the Haar basis for $L^{p}$.
\end{lemma}
\begin{proof}
Since $\displaystyle w\hyphen\lim_{n,\mathcal{U}}f_{n}=0$, by Lemma \ref{selconv}, we may assume that $f_{n}\to 0$ weakly as $n\to\infty$.

Let $(u_{j})_{j\geq 1}$ be the Haar basis for $L^{p}$. For $m\geq 1$, let $P_{m}$ be the projection from $L^{p}$ onto $\{u_{j}:1\leq j\leq m\}$, i.e.,
\[P_{m}u_{j}=\begin{array}{cc}\begin{cases}u_{j},&1\leq j\leq m\\0,&j>m\end{cases}.\end{array}\]
Since $(u_{j})_{j\geq 1}$ is a Schauder basis for $L^{p}$ \cite{Albiac}, $P_{m}\to I$ as $m\to\infty$ in the strong operator topology.

Choose $1\leq m(1)<n(1)<m(2)<n(2)<\ldots$ as follows:\\
There exists $m(1)\geq 1$ such that
\[\|f_{1}-P_{m(1)}f_{1}\|_{p}\leq\frac{1}{2}.\]
Since $f_{n}\to 0$ weakly and $P_{m(1)}$ has finite rank, there exists $n(1)>m(1)$ such that
\[\|P_{m(1)}f_{n}\|_{p}\leq\frac{1}{2},\quad n\geq n(1).\]
Suppose that $m(1)<n(1)<\ldots<m(k-1)<n(k-1)$ have been chosen. There exists $m(k)>n(k-1)$ such that
\begin{equation}\label{fn1}
\|f_{n}-P_{m(k)}f_{n}\|_{p}\leq\frac{1}{2^{k}},\quad n\leq n(k-1).
\end{equation}
Since $f_{n}\to 0$ weakly and $P_{m(k)}$ has finite rank, there exists $n(k)>m(k)$ such that
\begin{equation}\label{fn2}
\|P_{m(k)}f_{n}\|_{p}\leq\frac{1}{2^{k}},\quad n\geq n(k).
\end{equation}
Let $n(0)=0$. Since $\{[n(k)+1,n(k+1)]:k\geq 0\}$ is a partition of $\mathbb{N}$,
\[\mathbb{N}=\left(\bigcup_{k\text{ even}}[n(k)+1,n(k+1)]\right)\cup\left(\bigcup_{k\text{ odd}}[n(k)+1,n(k+1)]\right).\]
Since $\mathcal{U}$ is an ultrafilter, it contains exactly one of these two sets, say, the first one $\cup_{k\text{ even}}[n(k)+1,n(k+1)]$. Since $\mathcal{U}$ is selective, there exists $B\in\mathcal{U}$ such that $B\cap[n(2k)+1,n(2k+1)]$ is a singleton for each $k\geq 0$ and $B\subset\cup_{k\text{ even}}[n(k)+1,n(k+1)]$. Write $B=\{t(0),t(1),\ldots\}$ where $t(0)<t(1)<\ldots$. We have $n(2k)+1\leq t(k)\leq n(2k+1)$ so by (\ref{fn1}) and (\ref{fn2}),
\[\|f_{t(k)}-P_{m(2k+2)}f_{t(k)}\|_{p}\leq\frac{1}{2^{2k+2}}\]
and
\[\|P_{m(2k)}f_{t(k)}\|_{p}\leq\frac{1}{2^{2k}}.\]
So
\[\|f_{t(k)}-(P_{m(2k+2)}-P_{m(2k)})f_{t(k)}\|_{p}\leq\frac{1}{2^{2k+2}}+\frac{1}{2^{2k}}.\]
Let
\[g_{t(k)}=(P_{m(2k+2)}-P_{m(2k)})f_{t(k)},\quad k\geq 0,\]
and
\[g_{n}=0,\quad n\notin B.\]
Then $(f_{n})_{n,\mathcal{U}}=(g_{n})_{n,\mathcal{U}}$ and $(g_{n})_{n\geq 1}$ is a block sequence of $(u_{j})_{j\geq 1}$. Thus the result follows.
\end{proof}
\section{Proof of Theorem \ref{main}}
In this section, we first show that certain rank two operators are in the strong operator topology closure of $\{\rho_{\mathcal{U}}(\pi(T)):T\in B(L^{p})\}$. Then combining this with the results in Section 2 and 3, we obtain the proof of Theorem \ref{main}.

To begin, let us recall two classical results.
\begin{lemma}[\cite{Albiac}, Theorem 6.1.6]\label{unconditional}
Let $1<p<\infty$. Let $(u_{j})_{j\geq 1}$ be the Haar basis for $L^{p}$. Then there is a constant $C>0$ such that
\[\frac{1}{C}\left\|\sum_{j=1}^{r}a_{j}u_{j}\right\|\leq\left\|\sum_{j=1}^{r}\epsilon_{j}a_{j}u_{j}\right\|\leq C\left\|\sum_{j=1}^{r}a_{j}u_{j}\right\|,\]
for every $r\geq 1$, $\epsilon_{1},\ldots,\epsilon_{r}\in\{-1,1\}$ and scalars $a_{1},\ldots,a_{r}$.
\end{lemma}
\begin{lemma}[\cite{Albiac}, Theorem 6.2.14]\label{type}
Let $2<p<\infty$. There exists a constant $C>0$ such that
\[\mathbb{E}\left\|\sum_{i=1}^{r}\epsilon_{i}f_{i}\right\|_{p}\leq C\left(\sum_{i=1}^{r}\|f_{i}\|_{p}^{2}\right)^{\frac{1}{2}},\]
for every $r\geq 1$ and $f_{1},\ldots,f_{r}\in L^{p}$. The expectation is over $(\epsilon_{1},\ldots,\epsilon_{r})$ uniformly distributed on $\{-1,1\}^{r}$.
\end{lemma}
\begin{lemma}\label{2bound}
Let $2<p<\infty$. Let $\mathcal{U}$ be a selective nonprincipal ultrafilter on $\mathbb{N}$. Let $(f_{n})_{n,\mathcal{U}}\in\widehat{L^{p}}$. Then there are $C>0$, $g_{1},g_{2},\ldots\in L^{p}$ such that $(f_{n})_{n,\mathcal{U}}=(g_{n})_{n,\mathcal{U}}$ and
\[\left\|\sum_{i=1}^{r}a_{i}g_{i}\right\|_{p}\leq C\left(\sum_{i=1}^{r}|a_{i}|^{2}\right)^{\frac{1}{2}},\]
for every $r\geq 1$ and scalars $a_{1},\ldots,a_{r}$.
\end{lemma}
\begin{proof}
By Lemma \ref{blockhaar}, there are $g_{1},g_{2},\ldots\in L^{p}$ such that $(f_{n})_{n,\mathcal{U}}=(g_{n})_{n,\mathcal{U}}$ and $(g_{n})_{n\geq 1}$ is a block sequence of the Haar basis for $L^{p}$. Since $(g_{n})_{n\geq 1}$ is a block sequence of the Haar basis for $L^{p}$, by Lemma \ref{unconditional}, there is a constant $C>0$ such that
\[\left\|\sum_{i=1}^{r}a_{i}g_{i}\right\|_{p}\leq C\left\|\sum_{i=1}^{r}\epsilon_{i}a_{i}g_{i}\right\|_{p},\]
for every $r\geq 1$, $\epsilon_{1},\ldots,\epsilon_{r}\in\{-1,1\}$ and scalars $a_{1},\ldots,a_{r}$. Thus,
\[\left\|\sum_{i=1}^{r}a_{i}g_{i}\right\|_{p}\leq C\mathbb{E}\left\|\sum_{i=1}^{r}\epsilon_{i}a_{i}g_{i}\right\|_{p},\]
for every $r\geq 1$ and scalars $a_{1},\ldots,a_{r}$. So by Lemma \ref{type}, there are constants $C_{1},C_{2}>0$ such that
\[\left\|\sum_{i=1}^{r}a_{i}g_{i}\right\|_{p}\leq C_{1}\left(\sum_{i=1}^{r}|a_{i}|^{2}\|g_{i}\|_{p}^{2}\right)^{\frac{1}{2}}\leq C_{2}\left(\sum_{i=1}^{r}|a_{i}|^{2}\right)^{\frac{1}{2}},\]
for every $r\geq 1$ and scalars $a_{1},\ldots,a_{r}$.
\end{proof}
If $\mathcal{X}$ is a Banach space and $(x_{n}^{*})_{n,\mathcal{U}}\in\widehat{\mathcal{X}^{*}}$, then we can identify $(x_{n}^{*})_{n,\mathcal{U}}$ as an element in the dual of $\widehat{\mathcal{X}}$ via the duality relation defined in (\ref{duality}). So if $(x_{n})_{n,\mathcal{U}}\in\widehat{\mathcal{X}}$ and $(x_{n}^{*})_{n,\mathcal{U}}\in\widehat{\mathcal{X}^{*}}$, then $(x_{n})_{n,\mathcal{U}}\otimes(x_{n}^{*})_{n,\mathcal{U}}$ defines a rank one operator on $\widehat{\mathcal{X}}$. The following lemma provides a sufficient condition for this rank one operator to be in the strong operator topology (SOT) closure of $\{\rho_{\mathcal{U}}(\pi(T)):T\in B(\mathcal{X})\}$.
\begin{lemma}\label{1d}
Let $\mathcal{U}$ be a selective nonprincipal ultrafilter on $\mathbb{N}$. Let $\mathcal{X}$ be a reflexive Banach space. Let $(x_{n})_{n,\mathcal{U}}\in\widehat{\mathcal{X}}$ and $(x_{n}^{*})_{n,\mathcal{U}}\in\widehat{\mathcal{X}^{*}}$. Assume that for every $x\in \mathcal{X}$, the summation
\[\sum_{i=1}^{\infty}x_{i}^{*}(x)x_{i}\]
converges unconditionally. For each $A\in\mathcal{U}$, let
\[T_{A}=\sum_{i\in A}x_{i}\otimes x_{i}^{*}.\]
Then
\[\lim_{A\in\mathcal{U}}\rho_{\mathcal{U}}(\pi(T_{A}))=(x_{n})_{n,\mathcal{U}}\otimes(x_{n}^{*})_{n,\mathcal{U}},\]
where the convergence is in SOT and we treat $\mathcal{U}$ as a net with order defined by inverse inclusion of sets.
\end{lemma}
\begin{proof}
Without loss of generality, we may assume that $\|x_{n}\|\leq 1$ for all $n\geq 1$. Let $(y_{n})_{n,\mathcal{U}}\in\widehat{\mathcal{X}}$. By Lemma \ref{xy}, there exists $B=\{t(0),t(1),\ldots\}\in\mathcal{U}$, where $t(0)<t(1)<\ldots$, such that
\[|x_{t(j)}^{*}(y_{t(k)})|\leq\frac{1}{2^{\max(j,k)}},\quad j\neq k.\]
So for every $A\subset B$,
\begin{align*}
&\|[T_{A}^{\mathcal{U}}-(x_{n})_{n,\mathcal{U}}\otimes(x_{n}^{*})_{n,\mathcal{U}}](y_{n})_{n,\mathcal{U}}\|\\=&
\lim_{n,\mathcal{U}}\left\|\sum_{\substack{i\in A\\i\neq n}}x_{i}^{*}(y_{n})x_{i}\right\|\\\leq&\lim_{n,\mathcal{U}}\sum_{\substack{i\in B\\i\neq n}}|x_{i}^{*}(y_{n})|\|x_{i}\|\\\leq&\limsup_{k\to\infty}\sum_{\substack{j=1\\j\neq k}}^{\infty}|x_{t(j)}^{*}(y_{t(k)})|\|x_{t(j)}\|\\\leq&
\limsup_{k\to\infty}\sum_{\substack{j=1\\j\neq k}}^{\infty}\frac{1}{2^{\max(j,k)}}=0.
\end{align*}
Hence,
\[\lim_{A\in\mathcal{U}}[T_{A}^{\mathcal{U}}-(x_{n})_{n,\mathcal{U}}\otimes(x_{n}^{*})_{n,\mathcal{U}}](y_{n})_{n,\mathcal{U}}=0,\quad (y_{n})_{n,\mathcal{U}}\in\widehat{\mathcal{X}}.\]
Since $\rho_{\mathcal{U}}(\pi(T_{A}))$ is the restriction of $T_{A}^{\mathcal{U}}$ to $\widehat{\mathcal{X}}$, the result follows.
\end{proof}
In the next two lemmas, we apply Lemma \ref{1d} to show that certain rank one operators on $\widehat{L^{p}}$ are in the SOT closure of $\{\rho_{\mathcal{U}}(\pi(T)):T\in B(L^{p})\}$.
\begin{lemma}\label{nonR}
Let $2<p<\infty$. Let $\mathcal{U}$ be a selective nonprincipal ultrafilter on $\mathbb{N}$. Let $(f_{n})_{n,\mathcal{U}}\in\widehat{L^{p}}$ and $(f_{n}^{*})_{n,\mathcal{U}}\in\widehat{L^{q}}$ where $\displaystyle\frac{1}{p}+\frac{1}{q}=1$. Assume that $\displaystyle\sup_{n\geq 1}\|f_{n}^{*}\|_{2}<\infty$. Then $(f_{n})_{n,\mathcal{U}}\otimes(f_{n}^{*})_{n,\mathcal{U}}\in\{\rho_{\mathcal{U}}(\pi(T)):T\in B(L^{p})\}^{-SOT}$.
\end{lemma}
\begin{proof}
By Lemma \ref{blockhaar}, we may assume that $(f_{n}^{*})_{n,\mathcal{U}}$ is a block sequence of the Haar basis for $L^{q}$. Since the Haar basis consists of orthogonal functions, all the $f_{n}^{*}$ are orthogonal.

By Lemma \ref{2bound}, we may assume that there exists $C>0$ such that
\[\left\|\sum_{i=1}^{r}a_{i}f_{i}\right\|_{p}\leq C\left(\sum_{i=1}^{r}|a_{i}|^{2}\right)^{\frac{1}{2}},\]
for every $r\geq 1$ and scalars $a_{1},\ldots,a_{r}$. Thus, for every finite subset $F$ of $\mathbb{N}$,
\[\left\|\sum_{i\in F}f_{i}^{*}(x)f_{i}\right\|_{p}\leq C\left(\sum_{i\in F}|f_{i}^{*}(x)|^{2}\right)^{\frac{1}{2}},\quad x\in L^{p}.\]
Since all the $f_{n}^{*}$ are orthogonal and $\displaystyle\sup_{n\geq 1}\|f_{n}^{*}\|_{2}<\infty$, it follows that for every $x\in L^{p}\subset L^{2}$, the summation
\[\sum_{i=1}^{\infty}f_{i}^{*}(x)f_{i}\]
converges unconditionally. By Lemma \ref{1d}, the result follows.
\end{proof}
Recall that $R_{p}$ is defined in Lemma \ref{mproj}.
\begin{lemma}\label{R}
Let $2<p<\infty$. Let $\mathcal{U}$ be a selective nonprincipal ultrafilter on $\mathbb{N}$. Let $(f_{n})_{n,\mathcal{U}}\in\widehat{L^{p}}$ and $(f_{n}^{*})_{n,\mathcal{U}}\in\widehat{L^{q}}$. Assume that $(f_{n})_{n,\mathcal{U}}\in\mathrm{ran}\,R_{p}$ and $(f_{n}^{*})_{n,\mathcal{U}}\in\mathrm{ran}\,R_{q}$. Then $(f_{n})_{n,\mathcal{U}}\otimes(f_{n}^{*})_{n,\mathcal{U}}\in\{\rho_{\mathcal{U}}(\pi(T)):T\in B(L^{p})\}^{-SOT}$.
\end{lemma}
\begin{proof}
By Lemma \ref{disjoint}, we may assume that all the $f_{n}$ have disjoint supports and all the $f_{n}^{*}$ have disjoint supports. Moreover, we may also assume that $\|f_{n}\|_{p}=\|f_{n}^{*}\|_{q}=1$ for all $n\geq 1$. For every finite subset $F$ of $\mathbb{N}$,
\[\left\|\sum_{i\in F}f_{i}^{*}(x)f_{i}\right\|_{p}=\left(\sum_{i\in F}|f_{i}^{*}(x)|^{p}\right)^{\frac{1}{p}},\quad x\in L^{p}.\]
But
\[|f_{i}^{*}(x)|\leq\|f_{i}^{*}\|_{q}\|xI(\mathrm{supp}(f_{i}^{*}))\|_{p}=\|xI(\mathrm{supp}(f_{i}^{*}))\|_{p}.\]
Therefore,
\[\left\|\sum_{i\in F}f_{i}^{*}(x)f_{i}\right\|_{p}^{p}\leq\sum_{i\in F}\|xI(\mathrm{supp}(f_{i}^{*}))\|_{p}^{p}.\]
Since all the $\mathrm{supp}(f_{i}^{*})$ are disjoint, it follows that for every $x\in L^{p}$, the summation
\[\sum_{i=1}^{\infty}f_{i}^{*}(x)f_{i}\]
converges unconditionally. By Lemma \ref{1d}, the result follows.
\end{proof}
We now combine Lemmas \ref{nonR} and \ref{R} to obtain the following lemma.
\begin{lemma}\label{rankone}
Let $2<p<\infty$. Let $\mathcal{U}$ be a selective nonprincipal ultrafilter on $\mathbb{N}$. Let $(f_{n})_{n,\mathcal{U}}\in\widehat{L^{p}}$ and $(f_{n}^{*})_{n,\mathcal{U}}\in\widehat{L^{q}}$. Then the operator $(f_{n})_{n,\mathcal{U}}\otimes(f_{n}^{*})_{n,\mathcal{U}}-(I-R_{p})[(f_{n})_{n,\mathcal{U}}\otimes(f_{n}^{*})_{n,\mathcal{U}}]R_{p}$ on $\widehat{L^{p}}$ is in $\{\rho_{\mathcal{U}}(\pi(T)):T\in B(L^{p})\}^{-SOT}$.
\end{lemma}
\begin{proof}
We have
\begin{align*}
&(f_{n})_{n,\mathcal{U}}\otimes(f_{n}^{*})_{n,\mathcal{U}}-(I-R_{p})[(f_{n})_{n,\mathcal{U}}\otimes(f_{n}^{*})_{n,\mathcal{U}}]R_{p}\\=&
[(f_{n})_{n,\mathcal{U}}\otimes(f_{n}^{*})_{n,\mathcal{U}}](I-R_{p})+R_{p}[(f_{n})_{n,\mathcal{U}}\otimes(f_{n}^{*})_{n,\mathcal{U}}]R_{p}\\=&
(f_{n})_{n,\mathcal{U}}\otimes[(I-R_{q})(f_{n}^{*})_{n,\mathcal{U}}]+[R_{p}(f_{n})_{n,\mathcal{U}}]\otimes[R_{q}(f_{n}^{*})_{n,\mathcal{U}}],
\end{align*}
where the last equality follows from Lemma \ref{sa}.

For every $r>0$ and $n\geq 1$, let
\[h_{r,n}^{*}=f_{n}^{*}I(|f_{n}^{*}|\leq r)-\varphi_{r},\]
where $\displaystyle\varphi_{r}=w\hyphen\lim_{n,\mathcal{U}}f_{n}^{*}I(|f_{n}^{*}|\leq r)$. Then $|h_{r,n}^{*}|\leq 2r$ for all $n\geq 1$; $(h_{r,n}^{*})_{n,\mathcal{U}}\in\widehat{L^{q}}$, and $\displaystyle\lim_{r\to\infty}(h_{r,n}^{*})_{n,\mathcal{U}}=(I-R_{q})(f_{n}^{*})_{n,\mathcal{U}}$ (since $\displaystyle\lim_{r\to\infty}\varphi_{r}=0$ by Lemma \ref{rhat}).

By Lemma \ref{nonR}, $(f_{n})_{n,\mathcal{U}}\otimes(h_{r,n}^{*})_{n,\mathcal{U}}\in\{\rho_{\mathcal{U}}(\pi(T)):T\in B(L^{p})\}^{-SOT}$ for all $r>0$. Thus taking $r\to\infty$, we find that
\[(f_{n})_{n,\mathcal{U}}\otimes[(I-R_{q})(f_{n}^{*})_{n,\mathcal{U}}]\in\{\rho_{\mathcal{U}}(\pi(T)):T\in B(L^{p})\}^{-SOT}.\]
Also by Lemma \ref{R},
\[[R_{p}(f_{n})_{n,\mathcal{U}}]\otimes[R_{q}(f_{n}^{*})_{n,\mathcal{U}}]\in\{\rho_{\mathcal{U}}(\pi(T)):T\in B(L^{p})\}^{-SOT}.\]
Thus the result follows.
\end{proof}
\begin{proof}[Proof of Theorem \ref{main}]
It suffices to prove the result for $p>2$. For $p<2$, we can use duality and annihilation and apply the result for $p>2$.

By Lemma \ref{minvariant},
\[\{\rho_{\mathcal{U}}(\pi(T)):T\in B(L^{p})\}^{-WOT}\subset\{S\in B(\widehat{L^{p}}):S\mathcal{M}\subset\mathcal{M}\}.\]
This proves one direction. For the other direction, by Lemma \ref{rankone}, $S-(I-R_{p})SR_{p}\in\{\rho_{\mathcal{U}}(\pi(T)):T\in B(L^{p})\}^{-SOT}$ for every rank one operator $S$ on $\widehat{L^{p}}$. Since every operator on $\widehat{L^{p}}$ is the SOT limit of a net of finite rank operators on $\widehat{L^{p}}$, we have $S-(I-R_{p})SR_{p}\in\{\rho_{\mathcal{U}}(\pi(T)):T\in B(L^{p})\}^{-SOT}$ for every $S\in B(\widehat{L^{p}})$. By Lemma \ref{mproj}, $R_{p}$ is a projection onto $\mathcal{M}$ and so an operator on $\widehat{L^{p}}$ has $\mathcal{M}$ as an invariant subspace if and only if it has the form $S-(I-R_{p})SR_{p}$ for some $S\in B(\widehat{L^{p}})$. It follows that
$$\{S\in B(\widehat{L^{p}}):S\mathcal{M}\subset\mathcal{M}\}\subset\{\rho_{\mathcal{U}}(\pi(T)):T\in B(L^{p})\}^{-SOT}.\eqno\qedhere$$
\end{proof}
\begin{remark}
For $p<2$, the space $\mathcal{M}$ is the annihilator of the space $\mathcal{M}$ for the conjugate of $p$. In other words, for $p>2$, the space $\mathcal{M}$ is the range of $R_{p}$, whereas for $p<2$, the space $\mathcal{M}$ is the range of $I-R_{p}$.
\end{remark}
\section{Consequences}
The following corollary follows easily from Theorem \ref{main}.
\begin{corollary}\label{commutant}
Let $1<p<\infty$, $p\neq 2$. Let $\mathcal{U}$ be a selective nonprincipal ultrafilter on $\mathbb{N}$. Then the commutant of $\{\rho_{\mathcal{U}}(\pi(T)):T\in B(L^{p})\}$ in $B(\widehat{L^{p}})$ consists of scalar multiples of the identity operator.
\end{corollary}
As mentioned in Section 1, Calkin showed that \cite{Calkin} when $p=2$, the map $\rho_{\mathcal{U}}$ is an isometric $*$-representation so $\{\rho_{\mathcal{U}}(\pi(T)):T\in B(L^{2})\}$ is a $C^{*}$-subalgebra of $B((L^{2})^{\mathcal{U}})$. So by von Neumann's double commutant theorem, the double commutant of $\{\rho_{\mathcal{U}}(\pi(T)):T\in B(L^{2})\}$ in $B(\widehat{L^{2}})$ coincides with the WOT closure of $\{\rho_{\mathcal{U}}(\pi(T)):T\in B(L^{2})\}$.

Assume that $\mathcal{U}$ is selective. For $1<p<\infty$, $p\neq 2$, by Corollary \ref{commutant}, the double commutant of $\{\rho_{\mathcal{U}}(\pi(T)):T\in B(L^{p})\}$ in $B(\widehat{L^{p}})$ is  $B(\widehat{L^{p}})$, whereas by Theorem \ref{main}, $\{\rho_{\mathcal{U}}(\pi(T)):T\in B(L^{p})\}^{-WOT}=\{S\in B(\widehat{L^{p}}):S\mathcal{M}\subset\mathcal{M}\}$. Therefore, the double commutant of $\{\rho_{\mathcal{U}}(\pi(T)):T\in B(L^{p})\}$ in $B(\widehat{L^{p}})$ does not coincide with the WOT closure of $\{\rho_{\mathcal{U}}(\pi(T)):T\in B(L^{p})\}$.

The rest of this section is devoted to proving that the commutant of $B(L^{p})$ in its ultrapower may or may not be trivial depending on the ultrafilter if we assume the continuum hypothesis.
\begin{lemma}\label{rhocomm}
Let $1<p<\infty$, $p\neq 2$. Let $\mathcal{U}$ be a selective nonprincipal ultrafilter on $\mathbb{N}$. Let $P$ be the canonical projection from $(L^{p})^{\mathcal{U}}$ onto $L^{p}$. Then the commutant of $\{T^{\mathcal{U}}:T\in B(L^{p})\}$ in $B((L^{p})^{\mathcal{U}})$ is spanned by $P$ and $I-P$.
\end{lemma}
\begin{proof}
Suppose that $A\in B((L^{p})^{\mathcal{U}})$ commutes with $T^{\mathcal{U}}$ for all $T\in B(L^{p})$. By (\ref{tdecomp}), $T^{\mathcal{U}}=T\oplus\rho_{\mathcal{U}}(\pi(T))$ with respect to the decomposition $(L^{p})^{\mathcal{U}}=L^{p}\oplus\widehat{L^{p}}$. So for every compact operator $K$ on $L^{p}$, we have $K^{\mathcal{U}}=K\oplus 0$. Thus, $A$ commutes with $K\oplus 0$ for every $K\in K(L^{p})$. Since the identity operator on $L^{p}$ is the WOT limit of a sequence of compact operators on $L^{p}$, it follows that $A$ commutes with $I\oplus 0=P$. So we may write
\[A=A_{1}\oplus A_{2}\]
with respect to the decomposition $(L^{p})^{\mathcal{U}}=L^{p}\oplus\widehat{L^{p}}$. Since $A$ commutes with $T^{\mathcal{U}}$ for all $T\in B(L^{p})$, the operator $A_{1}$  is a scalar multiple of $I$, and $A_{2}$ commutes with $\rho_{\mathcal{U}}(\pi(T))$ for all $T\in B(L^{p})$. Thus by Lemma \ref{commutant}, $A_{2}$ is also a scalar multiple of $I$. Therefore, $A$ is in the span of $P$ and $I-P$.
\end{proof}
\begin{lemma}\label{nonrep}
Let $\mathcal{U}$ be a selective nonprincipal ultrafilter on $\mathbb{N}$. Let $\mathcal{X}$ be an infinite dimensional reflexive Banach space. Let $(T_{n})_{n\geq 1}$ be a bounded sequence in $B(\mathcal{X})$. Then
\[(T_{1},T_{2},\ldots)_{\mathcal{U}}\neq I\oplus 0\]
with respect to the decomposition $\mathcal{X}^{\mathcal{U}}=\mathcal{X}\oplus\widehat{\mathcal{X}}$.
\end{lemma}
\begin{proof}
Suppose by contradiction $(T_{1},T_{2},\ldots)_{\mathcal{U}}=I\oplus 0$. Let $(x_{k})_{k\geq 1}$ be a sequence in $\mathcal{X}$ such that $\|x_{k}\|=1$ for all $k\geq 1$ and $x_{k}\to 0$ weakly. Then
\[(T_{n}x_{k})_{n,\mathcal{U}}=(T_{1},T_{2},\ldots)_{\mathcal{U}}(x_{k})_{n,\mathcal{U}}=(x_{k})_{n,\mathcal{U}},\quad k\geq 1.\]
Thus $\displaystyle\lim_{n,\mathcal{U}}\|T_{n}x_{k}\|=1$ for all $k\geq 1$. For each $k\geq 1$, let
\[A_{k}=\{n\in\mathbb{N}:\|T_{n}x_{k}\|>\frac{1}{2}\}\in\mathcal{U}.\]
Since $\mathcal{U}$ is selective, by Lemma \ref{selectivesubseq}, there exists a sequence $(k_{n})_{n\geq 1}$ in $\mathbb{N}$ such that $k_{n}\to\infty$ and
\[\{n\in\mathbb{N}:\|T_{n}x_{k_{n}}\|>\frac{1}{2}\}\in\mathcal{U}.\]
Hence, $(T_{n}x_{k_{n}})_{n,\mathcal{U}}\neq 0$. But since $(x_{k_{n}})_{n,\mathcal{U}}\in\widehat{\mathcal{X}}$ and $(T_{1},T_{2},\ldots)_{\mathcal{U}}=0$ on $\widehat{\mathcal{X}}$, we have
\[(T_{n}x_{k_{n}})_{n,\mathcal{U}}=(T_{1},T_{2},\ldots)_{\mathcal{U}}(x_{k_{n}})_{n,\mathcal{U}}=0.\]
An absurdity follows.
\end{proof}
\begin{corollary}\label{selectivetrivial}
Let $1<p<\infty$, $p\neq 2$. Let $\mathcal{U}$ be a selective nonprincipal ultrafilter on $\mathbb{N}$. Suppose that $(T_{n})_{n\geq 1}$ is a bounded sequence in $B(L^{p})$ satisfying
\[\lim_{n,\mathcal{U}}\|T_{n}T-TT_{n}\|=0.\]
Then there exists a scalar $\lambda$ such that
\[\lim_{n,\mathcal{U}}\|T_{n}-\lambda I\|=0.\]
\end{corollary}
\begin{proof}
Since $(T_{1},T_{2},\ldots)_{\mathcal{U}}$ commutes with $T^{\mathcal{U}}$ for all $T\in B(L^{p})$, by Lemma \ref{rhocomm}, there exist scalars $\lambda_{1},\lambda_{2}$ such that
\[(T_{1},T_{2},\ldots)_{\mathcal{U}}=\lambda_{1}I\oplus\lambda_{2}I\]
with respect to the decomposition $(L^{p})^{\mathcal{U}}=L^{p}\oplus\widehat{L^{p}}$. By linear scaling and Lemma \ref{nonrep}, $\lambda_{1}=\lambda_{2}$. Thus the result follows.
\end{proof}
Recall that if $\mathcal{X}$ is a Banach space, the space of compact operators on $\mathcal{X}$ is denoted by $K(\mathcal{X})$. The following lemma should be well known. Roughly speaking it says that quasicentral approximate units exist for operators on reflexive Banach spaces.
\begin{lemma}[Compare to \cite{Boedihardjo2}, Lemma 2.2]\label{qcau}
Let $\mathcal{X}$ be a reflexive Banach space. Let $(K_{n})_{n\geq 1}$ be a bounded sequence in $K(\mathcal{X})$ converging to $I$ in SOT. Let $A_{1},\ldots,A_{r}\in B(\mathcal{X})$. Then there exists a sequence $(K_{n}')_{n\geq 1}$ in the convex hull of $\{K_{m}:m\geq 1\}$ converging to $I$ in SOT such that $\displaystyle\lim_{n\to\infty}\|K_{n}'A_{i}-A_{i}K_{n}'\|=0$ for all $i=1,\ldots,r$.
\end{lemma}
\begin{lemma}\label{flat}
Let $\mathcal{X}$ be a separable reflexive Banach space that has the bounded compact approximation property, i.e., there exists a bounded sequence $(K_{n})_{n\geq 1}$ in $K(\mathcal{X})$ converging to $I$ in SOT. Then there exists an ultrafilter $\mathcal{U}_{0}$ on $\mathbb{N}$ and a bounded sequence $(T_{n})_{n\geq 1}$ in $K(\mathcal{X})$ such that \[\lim_{n,\mathcal{U}}\|T_{n}x-x\|=0,\quad x\in \mathcal{X},\]
and
\[\lim_{n,\mathcal{U}}\|T_{n}T-TT_{n}\|=0,\quad T\in B(\mathcal{X}).\]
\end{lemma}
\begin{proof}
Let $\Lambda$ be the set of all sequence $a=(a_{j})_{j\geq 1}$ of rational numbers in $[0,1]$ such that only finite number of terms are nonzero and $\displaystyle\sum_{j=1}^{\infty}a_{j}=1$. Since $\Lambda$ is countably infinite, we may identify it with $\mathbb{N}$.

For each $a\in\Lambda$, let
\[T_{a}=\sum_{j=1}^{\infty}a_{j}K_{j}.\]
For every $x_{1},\ldots,x_{r}\in \mathcal{X}$, $A_{1},\ldots,A_{r}\in B(\mathcal{X})$ and $\epsilon>0$, the set
\[\{a\in\Lambda:\|T_{a}A_{i}-A_{i}T_{a}\|<\epsilon\text{ and }\|T_{a}x_{i}-x_{i}\|<\epsilon\text{ for all }1\leq i\leq r\}\]
is nonempty by Lemma \ref{qcau}. So these sets form a filter base on $\Lambda$ and thus are contained in an ultrafilter $\mathcal{U}_{0}$ on $\Lambda$. We have
\[\lim_{a,\mathcal{U}}\|T_{a}x-x\|=0,\quad x\in \mathcal{X},\]
and
$$\lim_{a,\mathcal{U}}\|T_{a}A-AT_{a}\|=0,\quad A\in B(\mathcal{X}).\eqno\qedhere$$
\end{proof}
\begin{remark}
The operators $K_{n}$ in Lemma \ref{flat} can be chosen to have norm 1 since every reflexive Banach space with the compact approximation property has the compact metric approximation property \cite[Proposition 1]{Cho}.
\end{remark}
\begin{corollary}\label{extension}
Let $1<p<\infty$. Let $\mathcal{U}$ be a nonprincipal ultrafilter on $\mathbb{N}$. Let $l^{\infty}(B(L^{p}))$ be the Banach algebra of bounded functions from $\mathbb{N}$ into $B(L^{p})$. Let $B(L^{p})^{\mathcal{U}}$ be the quotient of $l^{\infty}(B(L^{p}))$ by the ideal
\[c_{\mathcal{U}}(B(L^{p}))=\{(T_{1},T_{2},\ldots)\in l^{\infty}(B(L^{p})):\lim_{n,\mathcal{U}}\|T_{n}\|=0\}.\]
We may identify $B(L^{p})$ as a subalgebra of $B(L^{p})^{\mathcal{U}}$ via the map $T\mapsto(T,T,\ldots)+c_{\mathcal{U}}(B(L^{p}))$. Then the commutant of $B(L^{p})$ in $B(L^{p})^{\mathcal{U}}$ is trivial if $\mathcal{U}$ is selective; and there exists a nonprincipal ultrafilter $\mathcal{V}$ on $\mathbb{N}$ such that the commutant of $B(L^{p})$ in $B(L^{p})^{\mathcal{V}}$ is nontrivial.
\end{corollary}
\begin{proof}
For $1<p<\infty$, $p\neq 2$, the first assertion follows from Corollary \ref{selectivetrivial}, while the second assertion follows from Lemma \ref{flat}. For $p=2$, this was proved in \cite{Farah}.
\end{proof}
\section{Open problems}
\begin{problem}
Let $1<p<\infty$. Is $\{\rho_{\mathcal{U}}(\pi(T)):T\in B(L^{p})\}^{-WOT}$ always a reflexive operator algebra, i.e., if $S\in B((L^{p})^{\mathcal{U}})$ and $S\mathcal{N}\subset\mathcal{N}$ for all subspace $\mathcal{N}$ of $(L^{p})^{\mathcal{U}}$ that is invariant under $\rho_{\mathcal{U}}(\pi(T))$ for all $T\in B(L^{p})$, does $S$ necessarily have to be in $\{\rho_{\mathcal{U}}(\pi(T)):T\in B(L^{p})\}^{-WOT}$?
\end{problem}
Theorem \ref{main} gives an affirmative answer when $\mathcal{U}$ is selective. We also have an affirmative answer when $p=2$ and the scalar field is $\mathbb{C}$, since all von Neumann algebras are reflexive.
\begin{problem}
Let $1<p<\infty$. Assume that $\mathcal{U}$ is selective. Let $S\in\{\rho_{\mathcal{U}}(\pi(T)):T\in B(L^{p})\}^{-WOT}$. Does there exist $r>0$ such that $S\in\{\rho_{\mathcal{U}}(\pi(T)):T\in B(L^{p}),\,\|T\|\leq r\}^{-WOT}$?
\end{problem}
When $p=2$ and the scalar field is $\mathbb{C}$, we have an affirmative answer by Kaplansky density Theorem. But using the techniques in Section 4, it is not hard to see that we also have an affirmative answer when $p=2$ and the scalar field is $\mathbb{R}$.
\begin{problem}
Let $1<p<\infty$, $p\neq 2$. Characterize the operators $T\in B(L^{p})$ such that $T^{\mathcal{U}}$ commutes with $R_{p}$.
\end{problem}

\end{document}